\documentclass[12pt]{amsart}
\usepackage{amsmath,latexsym,amssymb}
\usepackage{enumerate}
\usepackage{amsthm}
\usepackage{epsfig,graphicx}
\usepackage{tikz}
\usetikzlibrary{arrows,automata}
\usepackage{mathrsfs}

\newtheorem{theorem}{Theorem}[section]

\newtheorem{corollary}[theorem]{Corollary}

\newtheorem{lemma}[theorem]{Lemma}
\newtheorem{remark}[theorem]{Remark}
\newtheorem*{remarks*}{Remarks}
\newtheorem{proposition}[theorem]{Proposition}
\newtheorem{conjecture}[theorem]{Conjecture}

\newcommand\NN{\mathbb{N}}
\newcommand\RR{\mathbb{R}}

\newcommand\CB{{\mathcal B}}

\newcommand\CF{{\mathcal F}}

\newcommand\CH{{\mathcal H}} 
\newcommand\CI{{\mathcal I}}

\newcommand\CS{{\mathcal S}}

\newcommand\pp{\mathbf{p}}
\newcommand\eps{\varepsilon}

\title[Random subsets of Cantor sets]{Random subsets of Cantor sets generated by trees of coin flips}

\author[Pieter Allaart]{Pieter Allaart$^*$}
\address[P. Allaart]{Mathematics Department, University of North Texas, 1155 Union Cir \#311430, Denton, TX 76203-5017, U.S.A.}
\email{allaart@unt.edu}
\thanks{$^*$ Corresponding author.}

\author{Taylor Jones}
\address[T. Jones]{Applied Mathematics \& Statistics Department, Johns Hopkins University, 3400 North Charles St, Baltimore, MD 21218-2683, U.S.A.}
\email{rjone197@jh.edu}

\begin{document}

\begin{abstract}
We introduce a natural way to construct a random subset of a homogeneous Cantor set $C$ in $[0,1]$ via random labelings of an infinite $M$-ary tree, where $M\geq 2$. The Cantor set $C$ is the attractor of an equicontractive iterated function system $\{f_1,\dots,f_N\}$ that satisfies the open set condition with $(0,1)$ as the open set.
%For a constant $c\in(0,\frac12]$, let $f_0(x):=cx+b_0$ and $f_1(x):=cx+b_1$, where $0\leq b_0<b_1\leq 1-c$ and $c+b_0\leq b_1$, so that the iterated function system (IFS) $\{f_0,f_1\}$ satisfies the open set condition with the open set being the interval $(0,1)$. The Cantor set $C$ is the attractor of this IFS, and every point $x\in C$ has a coding, i.e. a sequence $\xi=(\xi_1,\xi_2,\dots)$ such that $x=\lim_{n\to\infty}f_{\xi_1}\circ\dots\circ f_{\xi_n}(0)$.
For a fixed probability vector $(p_1,\dots,p_N)$, each edge in the infinite $M$-ary tree is independently labeled $i$ with probability $p_i$, for $i=1,2,\dots,N$. Thus, each infinite path in the tree receives a random label sequence of numbers from $\{1,2,\dots,N\}$. We define $F$ to be the (random) set of those points $x\in C$ which have a coding that is equal to the label sequence of some infinite path starting at the root of the tree. 

The set $F$ may be viewed as a statistically self-similar set with extreme overlaps, and as such, its Hausdorff and box-counting dimensions coincide. We prove non-trivial upper and lower bounds for this dimension, and obtain the exact dimension in a few special cases. For instance, when $M=N$ and $p_i=1/N$ for each $i$, we show that $F$ is almost surely of full Hausdorff dimension in $C$ but of zero Hausdorff measure in its dimension. 

For the case of two maps and a binary tree, we also consider deterministic labelings of the tree where, for a fixed integer $m\geq 2$, every $m$th edge is labeled $1$, and compute the exact Hausdorff dimension of the resulting subset of $C$.
\end{abstract}



\maketitle

\section{Introduction}

Consider an iterated function system (IFS) $\{f_1,\dots,f_N\}$ consisting of $N$ similar contractions $f_i: [0,1]\to[0,1]$, $i=1,2,\dots,N$ such that $f_i((0,1))\cap f_j((0,1))=\emptyset$ for all $i\neq j$. In other words, $\{f_1,\dots,f_N\}$ satisfies the open set condition with the open set being an interval. We assume furthermore that the IFS is equicontractive; that is, there is a constant $0<r\leq 1/N$ such that
\[
|f_i(x)-f_i(y)|=r|x-y| \qquad \mbox{for all $x,y\in[0,1]$, \ \ $i=1,2,\dots,N$}.
\]
%$f_0$ and $f_1$ have the same contraction ratio $0<c\leq 1/2$, so $f_0(x)=cx+b_0$ and $f_1(x)=cx+b_1$ for constants $b_0$ and $b_1$. Without loss of generality, we assume that $c+b_0\leq b_1$, so the image of $[0,1]$ under $f_0$ lies to the left of that under $f_1$. 
The attractor of the IFS $\{f_1,\dots,f_N\}$ is a nonempty compact set $C$ satisfying $C=\bigcup_{i=1}^N f_i(C)$. When $r<1/N$, $C$ is a Cantor set with
\[
\dim_H C=\dim_B C=-\frac{\log N}{\log r},
\]
where $\dim_H$ and $\dim_B$ denote Hausdorff and box-counting dimensions, respectively.
When $r=1/N$, $C$ is simply the unit interval $[0,1]$.

We are interested in certain random subsets of $C$ that are generated by trees of coin flips. The idea -- made precise in the next section but outlined here for the special case $N=2$ -- is simple: In a full infinite binary tree, we randomly label each edge with either $1$ with probability $p$ or $2$ with probability $1-p$, where $p\in(0,1)$ is a fixed parameter. Each path down the tree from the root then has a random infinite label sequence which is the coding of a point in $C$. We let $F$ denote the random subset of $C$ of all points obtained in this way. We show that when $p=1/2$, $F$ has almost surely the same Hausdorff dimension as $C$ itself, but with zero Hausdorff measure in its dimension. In particular, when $r=1/2$, our construction generates a random subset of the unit interval which  has Lebesgue measure zero but full Hausdorff dimension one almost surely. When $p\neq 1/2$, we obtain instead a random subset of $C$ with strictly smaller Hausdorff (and upper box-counting) dimension. For this biased case we are, unfortunately, only able to obtain rough (but non-trivial) bounds on the dimension. These are given in Theorem \ref{thm:dimension-bounds} below, which we prove in Sections \ref{sec:lower-bound} and \ref{sec:upper-bound}. In fact, we develops our bounds for arbitrary $N\geq 2$ and for random subsets obtained by random labelings of $M$-ary trees, where $M\geq 2$.

The random subset $F$ may be viewed as a statistically self-similar set with extreme overlaps, as we explain in the next section. Statistically self-similar sets were introduced in the 1980s by Falconer \cite{Falconer}, Graf \cite{Graf} and Mauldin and Williams \cite{Mauldin and Williams}, and have since been studied by many other authors. In almost all of these papers, some separation condition, such as the open set condition, is assumed. While there is a significant body of literature on dimension computations for deterministic self-similar sets with overlaps, little appears to be known about the random case in the absence of a separation condition. With this article, we hope to make a modest contribution to this line of research, although we did not manage to compute the exact dimension of $F$.

We observe that our setup is essentially the same as that of Benjamini {\em et al.}~\cite{BGS}. However, their principal focus was on the overlapping case (i.e. $r>1/N$), and instead of the random set $F$, they were mainly interested in the stochastically self-similar measure $\mu$ supported by $F$. This random measure $\mu$ is in fact precisely the measure $m_\omega$ that we use in the proof of the lower bound; see Section \ref{sec:lower-bound}. Benjamini {\em et al.}~do not give dimension estimates for $r<1/N$, and indeed, our results here are nearly disjoint from theirs.

We also consider, in Section \ref{sec:deterministic}, a deterministic version of our randomization scheme, corresponding to a sequence of coin flips in which every $m$th flip is heads, and compute exactly the Hausdorff dimension of the resulting subset of $C$. Interestingly, this dimension is not strictly monotone in $m$.

\subsection{Models of random fractals}

Following Hutchinson's result \cite{Hutchinson} on the Hausdorff dimension of self-similar fractals generated by an IFS satisfying the open set condition, there was a rapid rise in analogous results on a broad class of random constructions known as {\it statistically self-similar fractals}.  Below, we discuss a few models from the literature for constructing a random Cantor set in the line.
The randomness we are interested in resides in the process of deletion of intervals, as opposed to position or length of intervals (considered in \cite{Larsson, Koivusalo}).

Perhaps the best-known model of random Cantor sets begins with a partition of the unit interval into $M\ge2$ interior disjoint subintervals of length $1/M$, and a vector of probabilities $\mathbf{p}=(p_1,\dots,p_M)\in[0,1]^M$ (not necessarily a probability vector).
At stage $1$, the $i$th interval in the partition is either kept or discarded with probabilities $p_i$ and $1-p_i$ respectively. Within each non-discarded subinterval this process is repeated proportionally, and independently of past stages. The limit set, conditioned on non-extinction, results in a random Cantor-like fractal $F$, say.
In \cite{Falconer, Graf, Mauldin and Williams}, it was determined independently that
$\text{dim}_HF=\frac{1}{M}\log\big(\sum_{i=1}^{M}p_i\big)$
almost surely on $F\ne\varnothing$. 
This model was also used in \cite{Dekking and Karoly} to study the question of when the arithmetic difference of two random Cantor sets contains an interval.

Another model for constructing random Cantor sets, which removes the need for conditioning on non-extinction, is to fix a sequence of integers $N_1,N_2,\dots$ for which $1\le N_k\le M$ for all $k$. At each stage $k$ in the construction, and within each interval remaining at stage $k-1$, randomly and independently choose $N_k$ subintervals of size $M^{-k}$, then discard the rest.
This was used in \cite{Chen}, where the almost sure Hausdorff dimension was computed, as well as a close variation of it in \cite{Shmerkin and Suomala} in connection to tube null sets.

The random Cantor sets we consider here will also never go extinct, but our construction involves a complex form of interdependence which implies that the number of ``basic intervals" remaining at level $n$ is neither Markovian nor a martingale, even after appropriate scaling. As a result, traditional martingale approaches or branching arguments are not applicable. In fact, we could think of only one natural (random) mass distribution on our random Cantor subset, which led us to a lower bound that is (for the biased coin case) unfortunately significantly below the best upper bound we could find.

We point out that our construction also differs from the $V$-variable fractals of \cite{Barnsley}, and their generalization in \cite{Jarvenpaa} to random code tree fractals.

%In particular, we can think of our model as a branching process, see \cite{Lyons and Peres} and the examples therein, in which every particle produces either one or two offspring each generation. However, particles become more likely to produce two offspring if more past generations have only produced one. This connection becomes clear in the sequel.

\section{Notation and main result} \label{sec:main-result}

For definiteness we assume that $f_i([0,1])$ lies to the left of $f_j([0,1])$ when $i<j$. Further, for ease of presentation, we will also assume that each $f_i$ is orientation preserving; in other words, there are constants $b_1,\dots,b_N$ such that $f_i(x)=rx+b_i$, for $i=1,\dots,N$. However, this is not essential for our main results; see Remark \ref{rem:other-orientations} below.

Our probability space will be $\Omega:=\{1,\dots,N\}^\NN$ equipped with its Borel $\sigma$-algebra $\mathcal{B}$ and a probability measure $\mu_{\mathbf{p}}$ to be described below. We also write $\Omega_n:=\{1,\dots,N\}^n$ for $n\in\NN$, and $\Omega^*:=\bigcup_{n\in\NN}\Omega_n$. %We think of elements of $\Omega$ or $\Omega^*$ as sequences of ``coin flips". 
For any finite sequence $\omega^*=(\omega_1,\dots,\omega_n)\in\Omega^*$ with $|\omega^*|=n$ we define the basic open set, or cylinder,
$$[\omega^*]:=\{\omega\in\Omega:\omega|_n=\omega^*\},$$
where $\omega|_n$ denotes the restriction of $\omega$ to the first $n$ letters. Let $\CS:=\{[\omega^*]:\omega^*\in\Omega^*\}$. 

Denote by $\mathcal{P}$ the set of all probability vectors $\mathbf{p}=(p_1,\dots,p_N)\in[0,1]^N$ such that $\sum_{i=1}^N p_i=1$.
For $\pp=(p_1,\dots,p_N)\in\mathcal{P}$ we define a premeasure $\mu_\pp$ on $\CS$ by
$$\mu_\pp([\omega^*]):=\prod_{i=1}^N p_i^{\#\{k\leq n\::\:\omega_k^*=i\}}, \qquad \mbox{for}\ \omega^*=(\omega_1^*,\dots,\omega_n^*),$$
%p^{\#\{i\le|\omega^*|\::\:\omega_i^*=0\}}(1-p)^{\#\{i\le|\omega^*|\::\:\omega_i^*=1\}},$$
and observe that $\mu_\pp$ extends uniquely to a probability measure on $(\Omega,\CB)$ which we again denote by $\mu_\pp$.

Next, we consider a second integer parameter $M\geq 2$, and let
$$\CI_n:=\{1,\dots,M\}^n \quad (n\in\NN), \qquad \CI^*:=\bigcup_{n=0}^\infty \CI_n, \qquad \CI:=\{1,\dots,M\}^\infty.$$
Here we set $\CI_0:=\{\epsilon\}$, where $\epsilon$ denotes the empty word.
We define an order $<_*$ on $\CI^*$ in a natural way: For any pair of words $\mathbf{i}=(i_1,\dots,i_m)$ and $\mathbf{j}=(j_1,\dots,j_n)$ in $\CI^*$,
\begin{enumerate}
    \item $\mathbf{i}<_*\mathbf{j}$ if $m<n$, and
    \item $\mathbf{i}<_*\mathbf{j}$ if $m=n$ and $\mathbf{i}<_{\text{LEX.}}\mathbf{j}$ in the lexicographic order.
\end{enumerate}
We think of $<_*$ as ordering a binary tree from top to bottom, left to right.

Next, for each $\mathbf{i}\in\CI^*$ we define $\kappa:\CI^*\to\NN_0:=\NN\cup\{0\}$ by
$$\kappa(\mathbf{i})=\#\{\mathbf{j}\in\CI^*:\mathbf{j}<_*\mathbf{i}\}.$$
Note that $\kappa$ is well-defined as there are but finitely many words less than $\mathbf{i}$ in the $<_*$ order.
For instance, when $M=2$, the first few values of $\kappa$ are $\kappa(1)=0$, $\kappa(2)=1$, $\kappa(1,1)=2$, $\kappa(1,2)=3$, $\kappa(2,1)=4$, etc. 
Observe that $\kappa$ defines an order-preserving bijection from $(\CI^*,<_*)$ to $(\NN_0,<)$.

Now we define the collection of random variables $X_{\mathbf{i}}:\Omega^*\to\{1,\dots,N\}$ by
$$X_{\mathbf{i}}(\omega)=\omega_{\kappa(\mathbf{i})}, \qquad \mathbf{i}\in\CI^*.$$
These will be the random labels of the edges in the $M$-ary tree.

For any $\omega\in\Omega$, we define a subset of $C$ by
$$F(\omega):=\Bigg\lbrace x=\lim_{n\to\infty} f_{X_{\mathbf{i}|_1}(\omega)}\circ f_{X_{\mathbf{i}|_2}(\omega)}\circ \dots\circ f_{X_{\mathbf{i}|_n}(\omega)}([0,1]):\mathbf{i}=(i_1,i_2,\dots)\in\CI\Bigg\rbrace.$$
(We emphasize that due to the randomization, different choices of the path $\mathbf{i}$ could yield the same point $x$ in the limit set $F(\omega)$.)

\subsection{$F$ as the limit of a branching random walk}
To help visualize the random subset $F(\omega)$, the reader may wish to imagine a kind of branching random walk. We illustrate this for the case $M=N=2$. The walk begins with a single ancestor at time $0$ occupying the interval $[0,1]$. At each stage $n\geq 0$, each individual of generation $n$ splits into two descendants of generation $n+1$. If the individual occupies the interval $I$, its descendants, independently of each other, move either to the left subinterval $f_1(I)$ with probability $p_1$, or the right subinterval $f_2(I)$ with probability $p_2$. (Basically, this process is a ``branching random walk with exponentially decreasing steps", as in \cite{BGS}.)
Note that at stage $n$ there are $2^n$ individuals, collectively occupying some (random) subcollection of the intervals $I_{\mathbf{i}}, \mathbf{i}\in\{1,2\}^n$. Let $E_n$ denote the union of the occupied intervals; then $F(\omega)=\bigcap_{n=1}^\infty E_n$.

The difficulty in computing the dimension of $F(\omega)$ stems from the fact that, in order to understand the probability law of the branching random walk below stage $n$, it is not sufficient to know merely the number of occupied intervals at stage $n$, but in fact one needs to know the individual occupancy rates for each of the stage $n$ intervals. As a result, we are able to obtain the exact dimension only for a few special cases.

\subsection{$F$ as a statistically self-similar set} \label{subsec:self-similar}
We can also think of $F$ as a statistically (or stochastically) self-similar set with extreme overlaps. The general setup of statistically self-similar sets, introduced by Graf \cite{Graf}, is as follows: Let $Sim$ denote the set of all similarities in $\RR$, with the topology of uniform convergence, and let $\nu$ be a Borel probability measure on $Sim^M$. Consider the space $\widetilde{\Omega}:=(Sim^M)^{\CI^*}$, with the product measure $\nu^{\CI^*}$. For $\mathbf{i}\in \CI^*$ and $\tilde{\omega}\in\widetilde{\Omega}$, put $S_{\mathbf{i},j}(\tilde{\omega}):=\tilde{\omega}(\mathbf{i})_j$, $j=1,2,\dots,M$. When $\mathbf{i}=\epsilon$, we write simply $S_j$ instead of $S_{\epsilon,j}$.
Under $\nu^{\CI^*}$, the random $M$-tuples of maps $(S_{\mathbf{i},1},\dots,S_{\mathbf{i},M})$, $\mathbf{i}\in\CI^*$, are mutually independent and all have distribution $\nu$. Now define the random mapping $\pi_{\tilde{\omega}}: \CI\to\RR$ by 
\[
\pi_{\tilde{\omega}}(\mathbf{i}):=\lim_{n\to\infty} S_{\mathbf{i}|_1}(\tilde{\omega})\circ\dots \circ S_{\mathbf{i}|_n}(\tilde{\omega})(0),
\]
where $\mathbf{i}|_k$ denotes the prefix $i_1\dots i_k$ of the word $\mathbf{i}=i_1i_2\dots$. Finally, set $K(\tilde{\omega}):=\pi_{\tilde{\omega}}(\CI)$. The set $K(\tilde{\omega})$ is statistically self-similar in the sense that
\[
K(\tilde{\omega})=\bigcup_{j=1}^M S_j(\omega)(K_j(\tilde{\omega})),
\]
where $K_1,\dots,K_M$ are independent copies of $K(\tilde{\omega})$ that are also independent of $(S_1,\dots,S_M)$. If we take $\nu$ to be the product measure
\begin{equation} \label{eq:product-measure}
\nu=\bigotimes_{i=1}^M (p_1\delta_{f_1}+\dots+p_N \delta_{f_N}),
\end{equation}
then $K(\tilde{\omega})$ has the same probability distribution as $F(\omega)$. Note that for any $\tilde{\omega}\in\widetilde{\Omega}$ and any two words $\mathbf{i}, \mathbf{j}\in\CI^*$ of equal length, either $S_{\mathbf{i}}(\tilde{\omega})(0,1)\cap S_{\mathbf{j}}(\tilde{\omega})(0,1)=\emptyset$ or $S_{\mathbf{i}}(\tilde{\omega})(0,1)=S_{\mathbf{j}}(\tilde{\omega})(0,1)$. Thus, $F(\omega)$ is a statistically self-similar set with ``extreme overlaps".

It was shown by Falconer \cite{Falconer1} (see also \cite[Chapter 3]{Falconer-techniques}) that for any {\em deterministic} self-similar set, with or without overlaps, the Hausdorff and box-counting dimensions coincide. More recently, Liu and Wu \cite{Liu-Wu} proved that the same is true for statistically self-similar sets, under a mild condition which is trivially satisfied by the measure $\nu$ in \eqref{eq:product-measure}. Moreover, the dimension is almost surely constant. Therefore, we obtain the following result:

\begin{proposition}
There is a constant $a\geq 0$ such that, $\mu_\pp$-almost surely, $\dim_H F=\underline{\dim}_B F=\overline{\dim}_B F=a$.
\end{proposition}

Note that for all $\omega$ we have $F(\omega)\subseteq C$. On the other hand, $F$ can always be covered by at most $M^n$ intervals of length $r^{-n}$ for each $n$. Thus, we have the trivial upper bound
\begin{equation} \label{eq:trivial-upper-bound}
\overline{\dim}_B F(\omega)\leq \min\left\{-\frac{\log N}{\log r},-\frac{\log M}{\log r}\right\}.
\end{equation}
Our goal is to establish the following more precise dimension bounds:

\begin{theorem} \label{thm:dimension-bounds}
\begin{enumerate}[(a)]
\item {\em (Lower bound)} We have, $\mu_\pp$-almost surely,
\[
\dim_H F\geq \min\left\{-\frac{\log M}{\log r},\frac{\log(p_1^2+\dots+p_N^2)}{\log r}\right\}.
\]
\item {\em (Upper bound)} If 
\begin{equation} \label{eq:M-sandwich}
\prod_{i=1}^N p_i^{-p_i}\leq M\leq \left(\prod_{i=1}^N p_i\right)^{-1/N},
%\sum_{i=1}^N p_i\log p_i>-\log M>\frac{1}{N}\sum_{i=1}^N \log p_i,
\end{equation}
then there is a unique $\lambda\in[0,1]$ satisfying
\begin{equation} \label{eq:mu-equation}
\sum_{i=1}^N p_i^\lambda\log(Mp_i)=0,
\end{equation}
and
\begin{equation} \label{eq:upper-bound}
\overline{\dim}_B F\leq -\frac{\lambda\log M+\log\left(\sum_{i=1}^N p_i^\lambda\right)}{\log r}, \qquad \mbox{$\mu_\pp$-almost surely}.
\end{equation}
%for $\mu_\pp$-almost every $\omega\in\Omega$.
%\item If $\frac{1}{N}\sum_{i=1}^N \log p_i\geq -\log M$, then 
%$$\overline{\dim}_B F\leq -\frac{\log N}{\log r}.$$
%\end{enumerate}
\end{enumerate}
\end{theorem}

\begin{remark} \label{rem:upper-bound}
{\rm
(a) The upper bound in \eqref{eq:upper-bound} really is better than the trivial bound \eqref{eq:trivial-upper-bound}. To see this, set 
\[
\phi(x):=x\log M+\log\left(\sum_{i=1}^N p_i^x\right), \qquad x\in\RR,
\]
and observe that $\phi(0)=\log N$, $\phi(1)=\log M$;
\[
\phi'(x)=\log M+\frac{\sum p_i^x\log p_i}{\sum p_i^x}
\]
has exactly one zero namely at $x=\lambda$; $\phi'(0)=\log M+\frac{1}{N}\sum\log p_i\leq 0$; and $\phi'(1)=\log M+\sum p_i\log p_i\geq 0$; where the last two inequalities follow from \eqref{eq:M-sandwich}. Thus, $\phi(\lambda)\leq \min\{\log M,\log N\}$.

(b) Note that the inequalities in \eqref{eq:M-sandwich} are always satisfied when $M=N$, since
$$\sum_{i=1}^N p_i\log p_i\geq -\log N=\log\left(\frac{1}{N}\sum_{i=1}^N p_i\right)\geq \frac{1}{N}\sum_{i=1}^N\log p_i,$$
by concavity of the logarithm. But the range of values of $M$ satisfying \eqref{eq:M-sandwich} can be much larger. For instance, take $N=3$ and $\pp=(0.05,0.2,0.75)$. Then a direct computation yields that $\prod_{i=1}^N p_i^{-p_i}\approx 1.9886$ and $(\prod_{i=1}^N p_i)^{-1/N}\approx 5.1087$, so $M$ could be $2,3,4$ or $5$. 
}
\end{remark}

Unfortunately, if $M$ is outside the range \eqref{eq:M-sandwich}, our method yields nothing better than the trivial upper bound \eqref{eq:trivial-upper-bound}. However, sometimes that bound turns out to be the exact dimension:

\begin{corollary}[Large $N$, small $M$] \label{cor:small-M}
If
$$M\leq \min\left\{\prod_{i=1}^N p_i^{-p_i},\frac{1}{p_1^2+\dots+p_N^2}\right\},$$
then, $\mu_\pp$-almost surely,
$$\dim_H F=\dim_B F=-\frac{\log M}{\log r}.$$
\end{corollary}

For example, the conclusion of Corollary \ref{cor:small-M} holds when $N=3$, $M=2$ and $\pp=(0.2,0.2,0.6)$.

We observe that the inequality $M\leq (p_1^2+\dots+p_N^2)^{-1}$ can be satisfied only if either $M<N$, or $M=N$ and $p_i=1/N$ for each $i$, since it is easy to see that $\sum_{i=1}^N p_i^2\geq 1/N$, with equality if and only if $p_i=1/N$ for all $i$.

When $N=2$, the equation \eqref{eq:mu-equation} can be solved explicitly, and we get

\begin{corollary} \label{cor:two-maps}
Take $N=2$, and put $p:=p_1$. If $p(1-p)\leq M^{-2}$, then, $\mu_\pp$-almost surely,
$$\overline{\dim}_B F\leq \frac{\xi\log\xi+(1-\xi)\log(1-\xi)}{\log r},$$
where
$$\xi:=\xi(p):=\frac{\log Mp}{\log p-\log(1-p)}.$$
(When $M=2$ and $p=1/2$, we set $\xi(p):=1/2$.)
%Otherwise, we obtain only the trivial bound $\overline{\dim_B F}\leq -\log N/\log r$.
\end{corollary}

Observe that the condition $p(1-p)\leq M^{-2}$ is always satisfied when $M=2$. We illustrate the upper and lower bounds for the case $N=M=2$ in Figure \ref{fig:bounds}. The bounds are plotted as a function of $p:=p_1$. Observe that the upper bound tends to $0$ as $p$ tends to $0$ or $1$; however, it is still quite large even for values of $p$ very close to $0$ and $1$. This stems from the fact that the upper bound is the composition of two functions which both have vertical tangent lines at $0$ and $1$.

\begin{figure}
\begin{center}
\epsfig{file=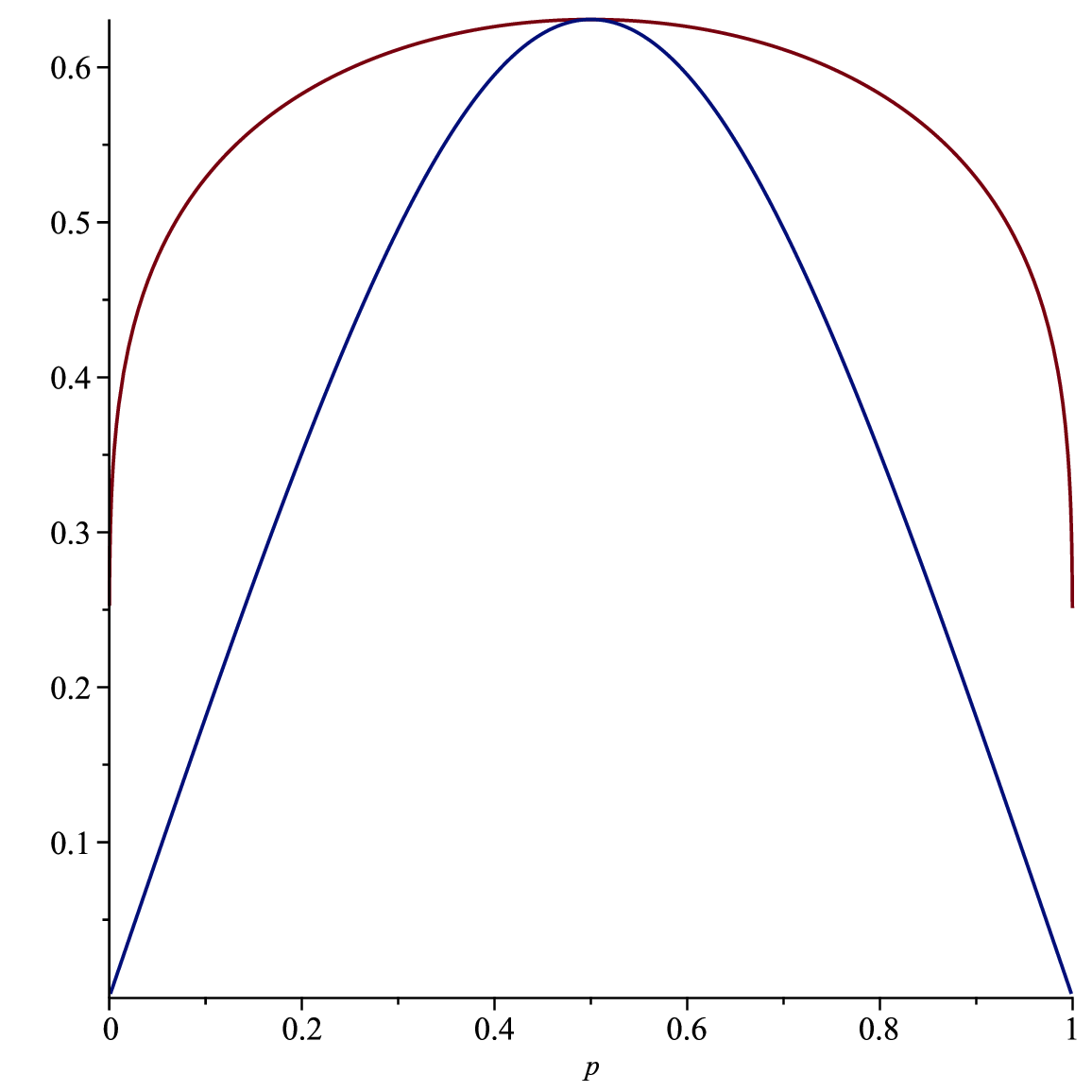, height=.25\textheight, width=.5\textwidth}
\end{center}
\caption{The upper and lower bounds of Corollary \ref{cor:two-maps}, with $N=M=2$.}
\label{fig:bounds}
\end{figure}

\begin{corollary} \label{cor:symmetric-case}
When $M\geq N$ and $p_i=1/N$ for each $i$, we have 
$$\dim_H F=\dim_B F=\dim_H C=-\frac{\log N}{\log r}\qquad\mbox{$\mu_\pp$-almost surely}.$$
\end{corollary}

%This follows since both bounds are equal to $-\frac{\log N}{\log r}$. 

\begin{remark}
{\rm
When $N=M$, the random subset $F$ may also be viewed as an orthogonal projection of a statistically self-similar set in the plane satisfying the open set condition. We thank an anonymous referee for pointing this out. %For ease of presentation, we take $N=M=2$ here. 
Define the maps
\[
g_{i,j}(x,y):=(f_i(x),f_j(y)), \qquad i,j\in\{1,2,\dots,M\}.
\]
%\begin{alignat*}{3}
%g_0(x,y)&=c(x,y)+(b_0,b_0), & \qquad g_1(x,y)&=c(x,y)+(b_0,b_1),\\
%g_2(x,y)&=c(x,y)+(b_1,b_0), & \qquad g_3(x,y)&=c(x,y)+(b_1,b_1).
%\end{alignat*}
These maps take the unit square onto $M^2$ non-overlapping subsquares of width $c$, arranged in $M$ rows and $M$ columns.
Now consider the setup from Subsection \ref{subsec:self-similar}, but with maps on $\RR^2$ instead of $\RR$. We take the probability measure $\nu$ to be
\[
\nu=\sum_{i_1,\dots,i_M}p_{i_1}\dots p_{i_M}\delta_{(g_{i_1,1},g_{i_2,2},\dots,g_{i_M,M})}.
\]
In essence, the measure $\nu$ picks out precisely one square in each row. Let $K(\tilde{\omega})$ be the resulting statistically self-similar set. It follows from the results in \cite{Falconer, Graf, Mauldin and Williams} that $K$ has almost sure Hausdorff and box-counting dimension $-\log M/\log r$. The projection of $K$ onto the $y$-axis is the full Cantor set $C$, while its projection onto the $x$-axis has the same probability distribution as our random subset $F$. Note that Marstrand's projection theorem implies that the projection of $K$ in the ``typical" direction (in the sense of Lebesgue measure) has Hausdorff dimension $-\log M/\log r=\dim_H C$. As a result, the horizontal direction in this context is ``typical", whereas the vertical direction is not, except when $p_i=1/M$ for each $i$.

%Let
%\[
%\CF_{i_1,\dots,i_N}:=\{g_{i_1,1},g_{i_2,2},\dots,g_{i_N,N}\}, \qquad (i_1,\dots,i_N)\in \CI_N.
%\]
%This gives a total of $N^N$ IFSs, each of which consists of $N$ maps and picks out precisely one square in each row.
%\[
%\CF_0:=\{g_{00},g_{01}\}, \qquad \CF_1:=\{g_{00},g_{11}\}, \qquad \CF_2:=\{g_{10},g_{01}\}, \qquad \CF_3:=\{g_{10},g_{11}\}.
%\]
%Now assign each {\em node} of the infinite $N$-ary tree a label from $\CI_N=(1,2,\dots,N)^N$, where $(i_1,\dots,i_N)\in \CI_N$ has probability $p_{i_1}\dots p_{i_N}$, and labels of different nodes are chosen independently.
%At each node of the tree, if this node has label $(i_1,\dots,i_N)$, we apply the IFS $\CF_{i_1,\dots,i_N}$. The limit set of this construction is a stochastically self-similar set $E$ in the plane of Hausdorff and box-counting dimension $-\log N/\log r$, whose projection onto the $y$-axis is the full Cantor set $C$ and whose projection onto the $x$-axis has the same probability distribution as our random subset $F$. Note that Marstrand's projection theorem implies that the projection of $E$ in the ``typical" direction (in the sense of Lebesgue measure) has Hausdorff dimension $-\log N/\log r=\dim_H C$. As a result, the horizontal direction in this context is ``typical", whereas the vertical direction is not, except when $p_i=1/N$ for each $i$.
}
\end{remark}

\section{Proof of the lower bound} \label{sec:lower-bound}

Recall that every basic interval at level $n$ in the construction of the deterministic Cantor set $C$ is given by $I_{\mathbf{i}}:=f_{i_1}\circ\dots \circ f_{i_n}([0,1])$ for some $\mathbf{i}\in\Omega_n$. %Precisely, $I_{\mathbf i}=f_{i_1}\circ\dots \circ f_{i_n}([0,1])$.
We first define a random probability measure on $F(\omega)$: For any $\mathbf{i}=(i_1,\dots,i_n)\in\Omega_n$ and $\mathbf{j}=(j_1,\dots,j_n)\in\CI_n$, define the event
\begin{equation} \label{eq:Aij}
A(\mathbf{i};\mathbf{j}):=\big\lbrace \omega\in\Omega: (X_{j_1}(\omega),X_{j_1,j_2}(\omega),\dots,X_{j_1,j_2,\dots,j_n}(\omega))=\mathbf{i}\big\rbrace.
\end{equation}
Then for $\mathbf{i}\in\Omega_n$, define
\begin{equation}
\label{random measure}
    m_\omega(I_{\mathbf{i}}):=\frac{1}{M^n}\sum_{\mathbf{j}\in\CI_n}1_{A(\mathbf{i};\mathbf{j})}(\omega).
\end{equation}

Intuitively, $m_\omega(I_{\mathbf i})$ is the relative number of branches labeled $\mathbf{i}$ in the tree at time $n$ by the realization $\omega$. 
It follows again from standard theorems that $m_\omega$ uniquely defines a measure on $F(\omega)$. 
We use this random measure along with the potential-theoretic method \cite[Theorem 4.13]{Falconer2} to prove the lower bound in Theorem \ref{thm:dimension-bounds}.

\begin{proof}[Proof of Theorem \ref{thm:dimension-bounds}, lower bound]
Fix 
\begin{equation} \label{eq:t-small-enough}
t<s(\pp):=\min\left\{-\frac{\log M}{\log r},\frac{\log(p_1^2+\dots+p_N^2)}{\log r}\right\}.
\end{equation}
We will derive an estimate for the $t$-energy of $m_\omega$, that is,
\[
\Phi_t(m_\omega):=\iint\limits_{[0,1]\times[0,1]}|x-y|^{-t}dm_\omega(x)dm_\omega(y).
\]

If $x$ and $y$ are two different points in $F(\omega)$, then there is a smallest basic interval (say $x\wedge y$) to which they both belong.
If $x\wedge y=I_{\mathbf{i}}$ with $|\mathbf{i}|=n$, we define $k(x,y)$ to be the smallest integer $k$ such that $x$ and $y$ are separated by at least one full basic interval of level $n+k$. Note that if $k(x,y)=k$, then %$k:=k(x,y)\geq 2$ since the IFS consists of only two maps, and
\[
|x-y|\geq r^{n+k},
\]
since each basic interval of level $n+k$ has length $r^{n+k}$. Without loss of generality we may assume that $x<y$. If $k=1$, then there are indices $i_{n+1}$ and $i_{n+1}'\in\{1,\dots,N\}$ such that $i_{n+1}'-i_{n+1}\geq 2$ and $x\in I_{\mathbf{i},i_{n+1}}$, $y\in I_{\mathbf{i},i_{n+1}'}$. If $k\geq 2$, there are two possibilities:
\begin{enumerate}[(i)]
\item $x\in I_{\mathbf{i},i_{n+1},N^{k-2},i_{n+k}}$ and $y\in I_{\mathbf{i},i_{n+1}+1,1^{k-2}}$, where $i_{n+1},i_{n+k}\in\{1,\dots,N-1\}$;
\item $x\in I_{\mathbf{i},i_{n+1},N^{k-2}}$ and $y\in I_{\mathbf{i},i_{n+1}+1,1^{k-2},i_{n+k}'}$, where $i_{n+1}\in\{1,\dots,N-1\}$ and $i_{n+k}'\in\{2,\dots,N\}$.
\end{enumerate}
Note that these two cases have some overlap; ignoring these overlaps leads at most to an overestimate of the energy.

We begin estimating the contribution of $\mathbf{i}$ to the $t$-energy of $m_\omega$, $\Phi_t^{\mathbf{i}}(m_\omega)$, by
\begin{align}
\Phi_t^{\mathbf{i}}(m_\omega) & :=\iint\limits_{\{(x,y):x\wedge y=I_{\mathbf{i}}\}}|x-y|^{-t}dm_\omega(x)dm_\omega(y) \label{eq:partial-energy} \\
&=\sum_{k=1}^\infty\ \ \iint\limits_{\substack{x\wedge y=I_{\mathbf{i}}, \\ k(x,y)=k}} |x-y|^{-t}dm_\omega(x)dm_\omega(y) \notag \\
&\leq 2\sum_{k=2}^\infty r^{-(n+k)t}\bigg(\sum_{i_{n+1}=1}^{N-1}\sum_{i_{n+k}=1}^{N-1}m_\omega\big(I_{\mathbf{i},i_{n+1},N^{k-2},i_{n+k}}\big)m_\omega\big(I_{\mathbf{i},i_{n+1}+1,1^{k-2}}\big) \label{eq:double-sum1} \\
&\qquad\qquad\qquad\qquad  + \sum_{i_{n+1}=1}^{N-1}\sum_{i_{n+k}'=2}^{N}m_\omega\big(I_{\mathbf{i},i_{n+1},N^{k-2}}\big)m_\omega\big(I_{\mathbf{i},i_{n+1}+1,1^{k-2},i_{n+k}'}\big)\bigg) \label{eq:double-sum2} \\
&\qquad\qquad + 2r^{-(n+1)t}\sum_{i_{n+1}=1}^{N-2}\sum_{i_{n+1}'=i_{n+1}+2}^N m_\omega\big(I_{\mathbf{i},i_{n+1}}\big)m_\omega\big(I_{\mathbf{i},i_{n+1}'}\big) \label{eq:sum-of-two-products}
\end{align}
We first estimate the expectation of the summands in \eqref{eq:double-sum1}.
By the definition of $m_\omega$, we have
\begin{gather*}
m_\omega\big(I_{\mathbf{i},i_{n+1},N^{k-2},i_{n+k}}\big)=\frac{1}{M^{n+k}}\sum_{\mathbf{j}\in \CI_{n+k}} 1_{A(\mathbf{i},i_{n+1},N^{k-2},i_{n+k};\mathbf{j})}(\omega),\\
m_\omega\big(I_{\mathbf{i},i_{n+1}+1,1^{k-2}}\big)=\frac{1}{M^{n+k}}\sum_{\mathbf{j'}\in \CI_{n+k-1}} 1_{A(\mathbf{i},i_{n+1}+1,1^{k-2};\mathbf{j'})}(\omega).
%m_\omega\big(I_{\mathbf{i}01^{k-2}0}\big)=\frac{1}{2^{n+k}}\sum_{\mathbf{j}\in\{0,1\}^{n+k}} 1_{A(\mathbf{i}01^{k-2}0,\mathbf{j})}(\omega),\\
%m_\omega\big(I_{\mathbf{i}10^{k-2}}\big)=\frac{1}{2^{n+k-1}}\sum_{\mathbf{j}\in\{0,1\}^{n+k-1}} 1_{A(\mathbf{i}10^{k-2},\mathbf{j})}(\omega),
\end{gather*}
Thus, letting $E_\pp$ denote the expectation operator associated with $\mu_\pp$,
\begin{align*}
E_\pp&\left[m_\omega\big(I_{\mathbf{i},i_{n+1},N^{k-2},i_{n+k}}\big)m_\omega\big(I_{\mathbf{i},i_{n+1}+1,1^{k-2}}\big)\right]\\
&=\frac{1}{M^{2n+2k-1}} \sum_{\mathbf{j}\in \CI_{n+k}} \sum_{\mathbf{j'}\in \CI_{n+k-1}} \mu_\pp\left(A(\mathbf{i},i_{n+1},N^{k-2},i_{n+k};\mathbf{j})\cap A(\mathbf{i},i_{n+1}+1,1^{k-2};\mathbf{j'})\right)\\
&\leq \frac{1}{M^{2n+2k-1}} M^k p_{i_{n+1}}p_N^{k-2}p_{i_{n+k}}M^{k-1}p_{i_{n+1}+1}p_1^{k-2} \sum_{\mathbf{j},\mathbf{j}'\in\CI_n} \mu_\pp\big(A(\mathbf{i};\mathbf{j})\cap A(\mathbf{i};\mathbf{j}')\big)\\
&\leq M^{-2n}(p_1 p_N)^{k-2} S_\pp(\mathbf{i}),
%E_p&\left[m_\omega\big(I_{\mathbf{i}01^{k-2}0}\big)m_\omega\big(I_{\mathbf{i}10^{k-2}}\big)\right]\\
%&=\frac{1}{2^{2n+2k-1}}\sum_{\mathbf{j}\in\{0,1\}^{n+k}}\sum_{\mathbf{j}'\in\{0,1\}^{n+k-1}} \mu_p\left(A(\mathbf{i}01^{k-2}0,\mathbf{j})\cap A(\mathbf{i}10^{k-2},\mathbf{j}')\right)\\
%&\leq\frac{1}{2^{2n+2k-1}}2^k p^2(1-p)^{k-2}\cdot 2^{k-1}(1-p)p^{k-2}\sum_{\mathbf{j},\mathbf{j}'\in\{0,1\}^n} \mu_p\big(A(\mathbf{i};\mathbf{j})\cap A(\mathbf{i},\mathbf{j}')\big)\\
%&=\frac{1}{4^n}p^k(1-p)^{k-1} \sum_{\mathbf{j},\mathbf{j}'\in\{0,1\}^n} \mu_p\big(A(\mathbf{i};\mathbf{j})\cap A(\mathbf{i},\mathbf{j}')\big)\\
%&=:\frac{1}{4^n}p^k(1-p)^{k-1} S_\pp(\mathbf{i}).
\end{align*}
where
\[
S_\pp(\mathbf{i}):=\sum_{\mathbf{j},\mathbf{j}'\in\CI_n} \mu_\pp\big(A(\mathbf{i};\mathbf{j})\cap A(\mathbf{i};\mathbf{j}')\big).
\]
Similarly, for the summands in \eqref{eq:double-sum2},
\[
E_\pp\left[m_\omega\big(I_{\mathbf{i},i_{n+1},N^{k-2}}\big)m_\omega\big(I_{\mathbf{i},i_{n+1}+1,1^{k-2},i_{n+k}'}\big)\right] \leq M^{-2n}(p_1 p_N)^{k-2} S_\pp(\mathbf{i}),
\]
and for the summands in \eqref{eq:sum-of-two-products},
\[
E_\pp\left[m_\omega\big(I_{\mathbf{i},i_{n+1}}\big)m_\omega\big(I_{\mathbf{i},i_{n+1}'}\big)\right]\leq M^{-2n} S_\pp(\mathbf{i}).
\]
Combining these estimates and noting that the double sums in \eqref{eq:double-sum1}-\eqref{eq:sum-of-two-products} have at most $N^2$ terms, we obtain
\begin{align}
\begin{split}
E_\pp\left[\Phi_t^{\mathbf{i}}(m_\omega)\right] &\leq 4N^2\sum_{k=2}^\infty r^{-(n+k)t}\cdot M^{-2n}(p_1 p_N)^{k-2} S_\pp(\mathbf{i})+2N^2 r^{-(n+1)t}M^{-2n}S_\pp(\mathbf{i}) \\
&\leq 4N^2 M^{-2n}S_\pp(\mathbf{i})\left(r^{-(n+2)t}\sum_{k=2}^\infty (r^{-t}p_1p_N)^{k-2}+r^{-(n+1)t}\right)\\
&=K(r^t M^2)^{-n}S_\pp(\mathbf{i})
%&=\frac{2}{4^n}c^{-(n+1)t}\sum_{k=2}^\infty \big[c^{-t}p(1-p)\big]^{k-1} S_\pp(\mathbf{i})\\
%&=K_{c,t,p}(4c^t)^{-n} S_\pp(\mathbf{i}),
\end{split}
\label{eq:E_p-calculation}
\end{align}
for a certain constant $K$ that depends on $r$ and $t$, but not on $\mathbf{i}$. Note that the last summation in \eqref{eq:E_p-calculation} converges because \eqref{eq:t-small-enough} implies
\[
r^{-t}p_1p_N<\frac{p_1p_N}{p_1^2+\dots+p_N^2}\leq \frac{p_1p_N}{p_1^2+p_N^2}\leq \frac12<1.
\]

To estimate the double sum $S_\pp(\mathbf{i})$, we introduce the sets
$$W_n(q):=\{(\mathbf{j},\mathbf{k})\in\CI_n\times\CI_n:|\mathbf{j}\wedge \mathbf{k}|=q\},\qquad q=0,1,\dots,n,$$ 
where $\mathbf{j}\wedge \mathbf{k}$ is the longest common prefix of $\mathbf{j}$ and $\mathbf{k}$.
Observe that for any $k\ge1$ we have $\#W_k(0)=(M-1)M^{2k-1}$. Furthermore, we have the recursive relation $\#W_n(q)=M^q\cdot\#W_{n-q}(0)$ for $q<n$. Thus, 
$$\#W_n(q)=M^q(M-1)M^{2n-2q-1}=(M-1)M^{2n-q-1}, \qquad q=0,1,\dots,n-1.$$
For $q=n$ we have $\#W_n(n)=M^n$. In all cases, then, $\#W_n(q)\leq M^{2n-q}$.

Before continuing with our previous calculation we define
\begin{equation} \label{eq:counts}
c_l(\mathbf{i}):=\#\{k\le n:i_k=l\}, \qquad \mathbf{i}\in\CI_n, \quad n\in\NN, \quad l=1,\dots,N.
\end{equation}
We also write, for $\mathbf{i}\in\CI_n$,
$$\mathbf{i}|^q:=(i_{q+1},\dots,i_n),$$
to be the suffix of $\mathbf{i}$ after the $q$th digit (as opposed to the prefix $\mathbf{i}|_q$ up to $q$).
Notice that $c_l(\mathbf{i}|_q)+c_l(\mathbf{i}|^q)=c_l(\mathbf{i})$. Now we have
\begin{equation} \label{exp of product of measure}
S_\pp(\mathbf{i})=\sum_{q=0}^{n}\sum_{(\mathbf{j},\mathbf{k})\in W_n(q)}\mu_\pp\big(A(\mathbf{i};\mathbf{j})\cap A(\mathbf{i};\mathbf{k})\big).
\end{equation}
    
Since the labels $(X_{j_1},\dots,X_{j_1,\dots,j_n})$ and $(X_{k_1},\dots,X_{k_1,\dots,k_n})$ of each pair of branches $(\mathbf{j},\mathbf{k})$ in $W_n(q)$ agree up to index $q$, and are independent after $q$, we have
\begin{equation*} %\label{eq:intersection-expression}
\mu_\pp\big(A(\mathbf{i};\mathbf{j})\cap A(\mathbf{i};\mathbf{k})\big) = \prod_{l=1}^N p_l^{c_l(\mathbf{i}|_q)}\cdot \prod_{l=1}^N p_l^{2c_l(\mathbf{i}|^q)}.
%p^{c_0(\mathbf{i}|_q)}(1-p)^{c_1(\mathbf{i}|_q)}p^{2c_0(\mathbf{i}|^q)}(1-p)^{2c_1(\mathbf{i}|^q)}
\end{equation*}
for all $(\mathbf{j},\mathbf{k})\in W_n(q)$.
Since there are at most $M^{2n-q}$ pairs of words in $W_n(q)$, (\ref{exp of product of measure}) gives
\[
S_\pp(\mathbf{i}) \leq \sum_{q=0}^{n}M^{2n-q}\prod_{l=1}^N p_l^{c_l(\mathbf{i}|_q)}\cdot \prod_{l=1}^N p_l^{2c_l(\mathbf{i}|^q)}.
\]
Writing $\mathbf{i}':=\mathbf{i}|_q$ and $\mathbf{i}'':=\mathbf{i}|^q$, and noting that summing over all $\mathbf{i}\in\Omega_n$ is the same as summing independently over $\mathbf{i}'\in\Omega_q$ and $\mathbf{i}''\in\Omega_{n-q}$, we obtain the estimate
\[
\sum_{\mathbf{i}\in\Omega_n}S_\pp(\mathbf{i}) \leq \sum_{q=0}^n M^{2n-q} \sum_{\mathbf{i}'\in\Omega_q}\prod_{l=1}^N p_l^{c_l(\mathbf{i}')}\cdot \sum_{\mathbf{i}''\in\Omega_{n-q}} \prod_{l=1}^N p_l^{2c_l(\mathbf{i}'')}.
\]
Now observe that
\[
%\sum_{\mathbf{i}'\in\Omega_q}\prod_{l=1}^N p_l^{c_l(\mathbf{i}')}=\sum_{k_1+\dots+k_N=q}\binom{q}{k_1,\dots,k_N}\prod_{i=1}^N p_i^{k_i}=1
\sum_{\mathbf{i}'\in\Omega_q}\prod_{l=1}^N p_l^{c_l(\mathbf{i}')}=\sum_{(i_1,\dots,i_q)\in\Omega_q}p_{i_1}\dots p_{i_q}=(p_1+\dots+p_N)^q=1
\]
and similarly,
\[
\sum_{\mathbf{i}''\in\Omega_{n-q}} \prod_{l=1}^N p_l^{2c_l(\mathbf{i}'')}=(p_1^2+\dots+p_N^2)^{n-q}.
\]
Hence,
\[
\sum_{\mathbf{i}\in\Omega_n}S_\pp(\mathbf{i}) \leq \sum_{q=0}^n M^{2n-q}(p_1^2+\dots+p_N^2)^{n-q}=M^n\sum_{j=0}^n\left[M(p_1^2+\dots+p_N^2)\right]^j.
\]
%\begin{align*}
%\sum_{\mathbf{i}\in\Omega_n}S_\pp(\mathbf{i}) & \leq \sum_{q=0}^n M^{2n-q} \sum_{\mathbf{i}'\in\Omega_q}\prod_{l=1}^N p_l^{c_l(\mathbf{i}')}\cdot \sum_{\mathbf{i}''\in\Omega_{n-q}} \prod_{l=1}^N p_l^{2c_l(\mathbf{i}'')}\\
%&=\sum_{q=0}^{n}M^{2n-q} \sum_{k_1+\dots+k_N=q}\binom{q}{k_1,\dots,k_N}\prod_{i=1}^N p_i^{k_i}\cdot \sum_{k_1+\dots+k_N=n-q}\binom{n-q}{k_1,\dots,k_N}\prod_{i=1}^N p_i^{2k_i}\\
%&=M^{n}\sum_{q=0}^n M^{n-q}(p_1^2+\dots+p_N^2)^{n-q}=M^n\sum_{r=0}^n\left[M(p_1^2+\dots+p_N^2)\right]^r.
%\end{align*}
Combining this with \eqref{eq:partial-energy} and \eqref{eq:E_p-calculation} yields
\begin{align*}
E_\pp\Big[\Phi_t(m_\omega)\Big] &=\sum_{n=0}^\infty \sum_{\mathbf{i}\in\Omega_n} E_\pp\left[\Phi_t^{\mathbf{i}}(m_\omega)\right]
\leq \sum_{n=0}^\infty K(r^t M^2)^{-n}\sum_{\mathbf{i}\in\Omega_n}S_\pp(\mathbf{i})\\
&\leq \sum_{n=0}^\infty K(r^t M)^{-n} \sum_{j=0}^n\left[M(p_1^2+\dots+p_N^2)\right]^j.
\end{align*}

We must consider two cases:

\bigskip

{\em Case 1}: $M(p_1^2+\dots+p_N^2)\geq 1$. Then
\begin{align*}
E_\pp\Big[\Phi_t(m_\omega)\Big] &\leq \sum_{n=0}^\infty K(n+1)(r^t M)^{-n}\left[M(p_1^2+\dots+p_N^2)\right]^n\\
&=\sum_{n=0}^\infty K(n+1)\big[r^{-t}(p_1^2+\dots+p_N^2)\big]^n<\infty,
\end{align*}
because \eqref{eq:t-small-enough} implies $r^{-t}<(p_1^2+\dots+p_N^2)^{-1}$.

\bigskip

{\em Case 2}: $M(p_1^2+\dots+p_N^2)<1$. Then $\sum_{j=0}^n\left[M(p_1^2+\dots+p_N^2)\right]^j$ is bounded by some constant $\tilde{K}$, and we get simply
\[
E_\pp\Big[\Phi_t(m_\omega)\Big] \leq K\tilde{K} \sum_{n=0}^\infty (r^{-t}M^{-1})^n<\infty,
\]
since \eqref{eq:t-small-enough} implies $r^{-t}<M$.

Therefore, for any $t<s(\pp)$ we have that $E_\pp\big[\Phi_t(m_\omega)\big]<\infty$, which means that $\Phi_t(m_\omega)<\infty$ for $\mu_\pp$-almost every $\omega\in\Omega$. Hence, $\dim_H F(\omega)\ge s(\pp)$ almost surely.
\end{proof}

\section{Proof of the upper bound} \label{sec:upper-bound}

\begin{proof}[Proof of Theorem \ref{thm:dimension-bounds}, upper bound]
Assume the inequalities \eqref{eq:M-sandwich} hold. Uniqueness of the zero $\lambda$ of \eqref{eq:mu-equation} follows from a version of Descartes' rule of signs: Arrange the $p_i$'s so that $p_1\leq p_2\leq\dots \leq p_N$, and make the substitution $u:=p_1^\lambda$. Then
\[
\sum p_i^\lambda \log(Mp_i)=\sum u^{\beta_i}\log(Mp_i)=:g(u),
\]
where $\beta_i:=\log p_i/\log p_1$. Note that the exponents $\beta_i$ decrease in $i$ whereas the coefficients $\log(Mp_i)$ increase in $i$, so there is at most one sign change in the coefficients. Descartes' rule of signs does not require the exponents to be integers, so there is at most one positive solution to $g(u)=0$.
 
For a word $\mathbf{i}=(i_1,\dots,i_n)\in\Omega_n$, let $a_{\mathbf{i}}$ be the probability under $\mu_\pp$ that at least one path of length $n$ starting at the root of the $M$-ary tree has label sequence $\mathbf{i}$. That is,
\begin{equation} \label{eq:a_i-definition}
a_{\mathbf{i}}:=\mu_\pp\left(\bigcup_{\mathbf{j}\in\CI_n}A(\mathbf{i};\mathbf{j})\right).
\end{equation}
Since the probability of $A(\mathbf{i};\mathbf{j})$ is the same for each path $\mathbf{j}$, we have
\begin{equation} \label{eq:simply-a_i-bound}
a_{\mathbf{i}}\leq M^n \prod_{l=1}^N p_l^{c_l(\mathbf{i})},
\end{equation}
where $c_l(\mathbf{i})$ was defined in \eqref{eq:counts}. Trivially, we also have $a_{\mathbf{i}}\leq 1$. 
Let $Z_n(\omega)$ denote the number of basic intervals at level $n$ in the construction of $C$ that intersect the limit set $F(\omega)$. Then
\begin{equation} \label{eq:expected-number}
E_\pp(Z_n)=\sum_{\mathbf{i}\in\Omega_n} a_{\mathbf{i}}.
\end{equation}
Our goal is to bound the growth rate of $E_\pp(Z_n)$. Note that
\begin{equation} \label{eq:expected-count-bound}
E_\pp(Z_n)\leq \sum_{k_1+\dots+k_N=n}\binom{n}{k_1,\dots,k_N}\min\left\{1,M^n\prod_{i=1}^N p_i^{k_i}\right\},
\end{equation}
by grouping together all label words of length $n$ that have the same digit frequencies. This is a rough estimate because it does not take into account the overlaps among the events $A(\mathbf{i};\mathbf{j})$, $\mathbf{j}\in\CI_n$. But understanding these overlaps appears to be a daunting task that we have not pursued further. We can, however, determine the exact exponential growth rate of the right hand side of \eqref{eq:expected-count-bound}.

Since the number of summands in \eqref{eq:expected-count-bound} is no greater than $n^{N-1}$ (a polynomial in $n$), it suffices to find the largest summand and determine its growth rate in $n$. To begin, write $k_i=\alpha_i n$ for $i=1,\dots,N$. Then we need to solve the optimization problem
\begin{equation} \label{eq:alpha-optimization}
\max\left\{\binom{n}{\alpha_1 n,\dots,\alpha_N n}\min\left\{1,\left(M\prod_{i=1}^N p_i^{\alpha_i}\right)^n\right\}: \alpha_i\in[0,1]\,\forall i, \ \sum_i\alpha_i=1\right\},
\end{equation}
where, if $\alpha_i n$ is not an integer, we define $(\alpha_i n)!$ to be $\Gamma(\alpha_i n+1)$, with $\Gamma$ denoting the gamma function. Now observe that
\[
M\prod_{i=1}^N p_i^{\alpha_i}=1 \quad\Longleftrightarrow \quad \sum_{i=1}^N \alpha_i\log p_i=-\log M,
\]
and this last equation determines a hyperplane in $\RR^N$. This hyperplane intersects the simplex
\[
\Delta:=\left\{(\alpha_1,\dots,\alpha_N)\in[0,1]^N: \sum_{i=1}^N \alpha_i=1\right\},
\]
and in fact, by \eqref{eq:M-sandwich}, the point $\pp=(p_1,\dots,p_N)$ lies on the side of the hyperplane where $M\prod_{i=1}^N p_i^{\alpha_i}\geq 1$, and the point $(1/N,\dots,1/N)$ lies on the opposite side. The functions
\[
\psi_1(\alpha_1,\dots,\alpha_N):=\binom{n}{\alpha_1 n,\dots,\alpha_N n}
\]
and
\[
\psi_2(\alpha_1,\dots,\alpha_N):=\binom{n}{\alpha_1 n,\dots,\alpha_N n}\prod_{i=1}^N p_i^{\alpha_i}
\]
are both unimodal on $\Delta$; $\psi_1$ is maximized at $(\alpha_1,\dots,\alpha_N)=(1/N,\dots,1/N)$; and $\psi_2$ is maximized when $(\alpha_1 n,\dots,\alpha_N n)$ is near the mode of the multinomial$(n;p_1,\dots,p_N)$ distribution, i.e. when $(\alpha_1,\dots,\alpha_N)\approx (p_1,\dots,p_N)$. The error in this approximation goes to zero as $n\to\infty$. Since the global maximum of both $\psi_1$ and $\psi_2$ lies on the ``wrong" side of the hyperplane $\sum_{i=1}^N \alpha_i\log p_i=-\log M$, the constrained maximum in \eqref{eq:alpha-optimization} must be attained on this hyperplane. By Stirling's approximation,
\[
\binom{n}{\alpha_1 n,\dots,\alpha_N n}=Q_n(\alpha_1,\dots,\alpha_n)\big(\alpha_1^{\alpha_1}\dots\alpha_N^{\alpha_N}\big)^{-n},
\]
where $Q_n$ is a function such that both $Q_n$ and $1/Q_n$ grow at most at a polynomial rate in $n$. Hence, to determine the exponential growth rate of the maximum in \eqref{eq:alpha-optimization}, we must solve the constrained optimization
\[
\max\left\{-\sum_{i=1}^N \alpha_i\log\alpha_i: (\alpha_1,\dots,\alpha_N)\in\Delta, \sum_{i=1}^N \alpha_i\log p_i=-\log M\right\}.
\]
This is a straightforward Lagrange multiplier problem; we find that the maximum is attained when
\[
\alpha_i=\frac{p_i^\lambda}{\sum_j p_j^\lambda}, \qquad i=1,\dots,N,
\]
where $\lambda$ satisfies the equation
\[
\frac{\sum p_i^\lambda\log p_i}{\sum p_i^\lambda}=-\log M,
\]
i.e. $\sum p_i^\lambda \log(Mp_i)=0$. A further calculation then yields that
\[
-\sum \alpha_i\log\alpha_i=\lambda\log M+\log\left(\sum p_i^\lambda\right).
\]
Denote this last expression by $\rho$. Then we find that, for given $\eps>0$ and sufficiently large $n$,
$E_\pp(Z_n)\leq (\rho+\eps)^n$. So
\[
\mu_\pp\big(Z_n>(\rho+2\eps)^n\big)\leq \frac{E_\pp(Z_n)}{(\rho+2\eps)^n}\leq \frac{(\rho+\eps)^n}{(\rho+2\eps)^n}.
\]
Hence the series $\sum_{n=1}^\infty \mu_\pp\big(Z_n>(\rho+2\eps)^n\big)$ converges, so by the Borel-Cantelli lemma, $Z_n\leq(\rho+2\eps)^n$ for all sufficiently large $n$ with probability one. Therefore, 
\[
\overline{\dim}_B F(\omega)\leq -\frac{\log(\rho+2\eps)}{\log r} \qquad\mbox{for $\mu_\pp$-almost every $\omega$},
\]
and letting $\eps\to 0$ along a discrete sequence gives
\[
\overline{\dim}_B F(\omega)\leq -\frac{\log\rho}{\log r} \qquad\mbox{for $\mu_\pp$-almost every $\omega$},
\]
as was to be shown.
\end{proof}

\begin{remark} \label{rem:quadratic-recursions}
{\rm
We briefly comment here on the difficulty of improving the upper bound. Consider the simplest case, $N=M=2$, and set $p_1=p$ and $p_2=1-p$. We can then write down a pair of quadratic recursions for the probabilities $a_{\mathbf{i}}$: For $n\geq 2$,
\begin{align}
\begin{split}
a_{1,i_2,\dots,i_n}&=2p(1-p)a_{i_2,\dots,i_n}+p^2(2a_{i_2,\dots,i_n}-a_{i_2,\dots,i_n}^2)\\
&=2pa_{i_2,\dots,i_n}-p^2 a_{i_2,\dots,i_n}^2,
\end{split}
\label{eq:a0}
\end{align}
and
\begin{align}
\begin{split}
a_{2,i_2,\dots,i_n}&=2p(1-p)a_{i_2,\dots,i_n}+(1-p)^2(2a_{i_2,\dots,i_n}-a_{i_2,\dots,i_n}^2)\\
&=2(1-p)a_{i_2,\dots,i_n}-(1-p)^2 a_{i_2,\dots,i_n}^2.
\end{split}
\label{eq:a1}
\end{align}
For instance, \eqref{eq:a0} can be understood by looking at the three cases in Figure \ref{three cases general}.

\begin{figure}[h]
\centering
\begin{tikzpicture}
    \tikzstyle{level 1}=[level distance=10mm,sibling distance=2.75cm]
    \tikzstyle{solid node}=[circle,draw,inner sep=1.2,fill=black];

    \node (0) [solid node]{ }
        child{ node[solid node, label = below:{$\substack{\text{at least one branch}\\ \text{of length }n-1\\ \text{labeled $i_2,\dots,i_n$ here}}$}]{ } edge from parent node [left] {1}}
        child{ node[solid node]{ } edge from parent node [right] {2}};
\end{tikzpicture}\:\:\:\:\:
\begin{tikzpicture}
    \tikzstyle{level 1}=[level distance=10mm,sibling distance=2.75cm]
    \tikzstyle{solid node}=[circle,draw,inner sep=1.2,fill=black];

    \node (0) [solid node]{ }
        child{ node[solid node]{ } edge from parent node [left] {2}}
        child{ node[solid node, label = below:{$\substack{\text{at least one branch}\\ \text{of length }n-1\\ \text{labeled $i_2,\dots,i_n$ here}}$}]{ } edge from parent node [right] {1}};
\end{tikzpicture}\:\:
\begin{tikzpicture}
    \tikzstyle{level 1}=[level distance=10mm,sibling distance=2.75cm]
    \tikzstyle{solid node}=[circle,draw,inner sep=1.2,fill=black];

    \node (0) [solid node]{ }
        child{ node[solid node, label = below:{$\substack{\text{either at least one} \\ \text{branch labeled}\\ \text{$i_2,\dots,i_n$ here or...}}$}]{ } edge from parent node [left] {1}}
        child{ node[solid node, label = below:{$\substack{\text{...at least one}\\ \text{branch labeled}\\ \text{$i_2,\dots,i_n$ here}}$}]{ } edge from parent node [right] {1}};
\end{tikzpicture}
\caption{Three possibilities if a branch labeled $(1,i_2,\dots,i_n)$ is in the tree at time $n$.}
\label{three cases general}
\end{figure}
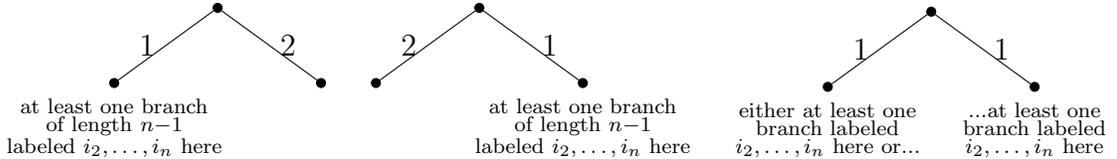

Ignoring the quadratic terms in the recursions would simply lead us back to \eqref{eq:simply-a_i-bound}. But a full understanding of their effects, averaged out over all paths $\mathbf{i}$, would entail a delicate analysis of the joint dynamics of the two quadratic maps $x\mapsto px(2-px)$ and $x\mapsto (1-p)x(2-(1-p)x)$, which appears to be extremely  complicated. (For one thing, better bounds on $a_{\mathbf{i}}$ would depend not only on the number of $1$'s and $2$'s in $\mathbf{i}$, but also on the order in which the digits appear.)
}
\end{remark}

\begin{remark} \label{rem:other-orientations}
{\rm
While we assumed in Theorem \ref{thm:dimension-bounds} that the maps $f_1,\dots,f_N$ are orientation preserving, the same result holds when one or more maps are instead made orientation-reversing. We briefly indicate here how the proof of the lower bound would need to be modified. For ease of presentation, consider again the case $N=M=2$ and suppose $f_1(x)=-cx+b_1$ and $f_2(x)=cx+b_2$, where we still assume that $f_1([0,1])$ lies to the left of $f_2([0,1])$. In this case, (i) and (ii) at the beginning of the proof of the lower bound in Theorem \ref{thm:dimension-bounds} become, for $k\geq 3$:
\begin{enumerate}[(i)]
\item $x\in I_{\mathbf{i},1,1,2^{k-3},1}$ and $y\in I_{\mathbf{i},2,1,2^{k-3}}$;
\item $x\in I_{\mathbf{i},1,1,2^{k-3}}$ and $y\in I_{\mathbf{i},2,1,2^{k-3},1}$.
\end{enumerate}
For $k=2$ we have $x\in I_{\mathbf{i},1,2}$ and $y\in I_{\mathbf{i},2}$, or $x\in I_{\mathbf{i},1}$ and $y\in I_{\mathbf{i},2,2}$. The constant $K$ from \eqref{eq:E_p-calculation} is now replaced with the sum
\[
\sum_{k=3}^\infty r^{-kt}p^4(1-p)^{2(k-3)+1}
\]
(plus an unimportant term for $k=2$), which converges for 
$$t<s(p):=\frac{\log\big(p^2+(1-p)^2\big)}{\log r}$$ 
because the latter implies
\[
r^{-t}(1-p)^2<\frac{(1-p)^2}{p^2+(1-p)^2}<1.
\]
Other combinations of orientations can be dealt with similarly.
}
\end{remark}

\section{Hausdorff measure in the symmetric case} \label{sec:symmetric}

Throughout this section, we assume that we are in the symmetric case where all of the maps $f_1,\dots,f_N$ are chosen with the same probability; that is, $\pp=(1/N,\dots,1/N)$.

Corollary \ref{cor:symmetric-case} states that if $M\geq N$, then the Hausdorff and box-counting dimensions of $F$ are almost surely equal to $-\log N/\log r$, the dimension of $C$. In this section we consider the Hausdorff {\em measure} of $F$ in the critical dimension. By using more information about the overlaps between the sets $A(\mathbf{i};\mathbf{j})$ for $\mathbf{j}\in\CI_n$, we can say something slightly stronger and show the following:

\begin{proposition} \label{prop:symmetric-case-measure-zero}
Assume $M=N$ and $\pp=(1/N,\dots,1/N)$. Then $\CH^s(F)=0$ almost surely with respect to $\mu_\pp$, where $s=-\log N/\log r$.
\end{proposition}

(By this we mean that there is a set of $\omega$'s of full measure under $\mu_\pp$ for which $\CH^s(F(\omega))=0$; note that it is not obvious {\em a priori} that the function $\omega\mapsto \CH^s(F(\omega))$ is measurable. We will therefore bound it by a function whose measurability is not in doubt and which takes the value $0$ for $\mu_\pp$-almost all $\omega$.)

As a consequence of Proposition \ref{prop:symmetric-case-measure-zero}, we point out that when $M=N$ and $r=1/N$, our construction generates a random subset of the interval which has Lebesgue measure zero but full Hausdorff dimension one almost surely.

Recall the definition of $a_{\mathbf{i}}$ from \eqref{eq:a_i-definition}. We also continue to use $Z_n$ for the number of basic intervals at level $n$ in the construction of $C$ needed to cover $F$. We first prove the following dichotomy:

\begin{lemma} \label{lem:dichotomy}
%Assume $\pp=(1/N,\dots,1/N)$.
\begin{enumerate}[(i)]
\item If $M\leq N$, then $E_\pp(Z_n)/N^n\to 0$.
\item If $M>N$, then $E_\pp(Z_n)/N^n\to \gamma$, for some number $\gamma\in(0,1)$.
\end{enumerate}
\end{lemma}

\begin{proof}
By symmetry, $a_{\mathbf{i}}$ is the same for all $\mathbf{i}\in\Omega_n$, so we just consider the label sequence $\mathbf{i}=1^n$. Set $\pi_n:=a_{1^n}$. Then, by reasoning similar to that in Remark \ref{rem:quadratic-recursions}, $\pi_n$ satisfies the recursion
\begin{align*}
\pi_n&=\sum_{k=0}^M\binom{M}{k}\left(\frac{1}{N}\right)^k\left(1-\frac{1}{N}\right)^{M-k}\left(1-(1-\pi_{n-1})^k\right)\\
&=1-\left(1-\frac{\pi_{n-1}}{N}\right)^M.
\end{align*}
Clearly $\pi_n$ is decreasing in $n$, so its limit as $n\to\infty$ exists. Denote this limit by $\gamma$; then $\gamma$ is a root of the equation
\[
x=1-\left(1-\frac{x}{N}\right)^M=:h(x).
\]
Observe that
\[
h'(x)=\frac{M}{N}\left(1-\frac{x}{N}\right)^{M-1}, \qquad h''(x)=-\frac{M(M-1)}{N^2}\left(1-\frac{x}{N}\right)^{M-2}.
\]
This shows that $h$ is strictly concave on $[0,1]$, and if $M\leq N$, then $h'\leq 1$ on $[0,1]$. Since furthermore, $h(0)=0$, it follows that when $M\leq N$, $h$ has a unique fixed point $x=0$, and hence $\pi_n\to 0$ as $n\to\infty$. Part (i) of the lemma now follows since
\begin{equation} \label{eq:Z_n-limit}
E_\pp(Z_n)=\sum_{\mathbf{i}\in\Omega_n} a_{\mathbf{i}}=N^n a_{1^n}=N^n \pi_n.
\end{equation}

On the other hand, if $M>N$, then $h'(0)=M/N>1$ whereas $h(1)<1$, so $h$ has an additional fixed point $\gamma\in(0,1)$. In this case, the fixed point $x=0$ is repelling, while the fixed point $\gamma$ is attracting. Thus, $\pi_n\to\gamma$ as $n\to\infty$, and part (ii) of the lemma follows by \eqref{eq:Z_n-limit}.
\end{proof}

We also need the following standard result from probability, of which we omit the proof.

\begin{lemma} \label{lem:decreasing-rv-sequence}
Let $\{X_n\}_{n\in\NN}$ be a sequence of random variables defined on a probability space $(\Omega,\CF,P)$. Suppose $E(X_n)\to 0$ as $n\to\infty$, and for all $n\in\NN$ and all $\omega\in\Omega$ we have $0\le X_{n+1}(\omega)\le X_n(\omega)$. Then $X_n\to 0$ almost surely.
\end{lemma}

\begin{proof}[Proof of Proposition \ref{prop:symmetric-case-measure-zero}]
Given $\delta>0$, choose $n_0$ so large that $r^{n_0}<\delta$. Then since $F(\omega)$ may be covered by $Z_n(\omega)$ intervals of length $r^{n}$, we have for all $n\ge n_0$ and all $\omega\in\Omega$
$$\CH_\delta^s(F(\omega))\le r^{sn}Z_n(\omega)=N^{-n}Z_n(\omega).$$
(Recall $s=-\log N/\log r$).
Note that $Z_{n+1}(\omega)\le NZ_n(\omega)$ for all $\omega$. 
Therefore $N^{-n}Z_n$ is decreasing, and the limit in the following inequality exists for all $\omega$:
$$\CH_\delta^s(F(\omega))\le\lim_{n\to\infty}N^{-n}Z_n(\omega)=:Y(\omega).$$
As this holds for all $\delta>0$, we have
$\CH^s(F(\omega))\le Y(\omega)$
for all $\omega\in\Omega$. Furthermore,
$$E_\pp\Big(N^{-n}Z_n\Big)=\pi_n\to 0 \qquad \text{as }n\to\infty$$
by Lemma \ref{lem:dichotomy} (i). This implies $\mu_\pp(Y=0)=1$ by applying Lemma \ref{lem:decreasing-rv-sequence} with $X_n=N^{-n}Z_n$, $P=\mu_\pp$ and $(\Omega,\CF)$, completing the proof.
\end{proof}

\begin{remark}
{\rm
By a very similar argument, we can show also that $\CH^{s'}(F)=0$ almost surely when $M<N$, where $s'=-\log M/\log r$. Thus, at least in the symmetric case of Corollary \ref{cor:small-M}, we also have that the Hausdorff measure of $F$ in the critical dimension is zero. Unfortunately, we do not know how to prove this, or whether it is in fact true, in the non-symmetric case.
}
\end{remark}

While we did not use Lemma \ref{lem:dichotomy} (ii) in the above proof, it does serve to inform the following conjecture:

\begin{conjecture} \label{con:positive-measure}
When $M>N$ and $\pp=(1/N,\dots,1/N)$, $\CH^s(F)>0$ almost surely with respect to $\mu_\pp$, where $s=-\log N/\log r$.
\end{conjecture}

\begin{remark}
{\rm
When $N=2$ and $M\geq 4$, something stronger than Conjecture \ref{con:positive-measure} is true: With $\mu_\pp$-probability 1, $F$ contains a similar copy of the full Cantor set $C$ (or contains an interval in case $r=1/2$), and $F$ is even equal to all of $C$ with positive probability. As pointed out in \cite{BGS}, this is a consequence of \cite[Corollary 6.2]{Benjamini-Kesten}. The argument there is easily extended to obtain the same conclusion for arbitrary $N\geq 2$ and $M\geq 2N$. This still leaves a gap in our knowledge: What happens when $N<M<2N$?
}
\end{remark}

\begin{remark}
{\rm
In view of Proposition \ref{prop:symmetric-case-measure-zero} it is natural to ask whether in the case $M=N$ there exists a gauge function $\varphi$ (that is, a strictly increasing continuous function $\varphi:[0,\infty)\to[0,\infty)$ with $\varphi(0)=0$) such that
\begin{equation} \label{eq:gauge-function}
0<\mathcal{H}^\varphi(F(\omega))<\infty \qquad\mbox{for $\mu_\pp-$ almost every $\omega$},
\end{equation}
where for a set $E$,
\[
\mathcal{H}^\varphi(E):=\lim_{\delta\searrow 0} \inf\left\{\sum_i \varphi(|U_i|): (U_i)\ \mbox{is a $\delta$-cover of $E$}\right\}.
\]
For instance, for $N=M=2$ we can show that
\begin{equation*} %\label{eq:p_n-sandwich}
\frac{1}{1+n}\le \pi_n\le\frac{4}{4+n} \qquad \forall\,n\geq 0.
\end{equation*}
We expect that similar inequalities hold in general when $N=M$.
These inequalities, together with $E_\pp(Z_n)=N^n \pi_n$ suggest that a good candidate is the function $\varphi(x)=x^s\log(1/x)$, where $s=-\log N/\log r$. Unfortunately, we have been unable to prove either inequality in \eqref{eq:gauge-function} for this $\varphi$.
}
\end{remark}

\section{Deterministic Case} \label{sec:deterministic}

In this section we consider a simple deterministic scheme for labeling the tree, and compute the Hausdorff dimension of the resulting subsets of $C$. For the sake of simplicity, we take $N=M=2$ and denote the maps of the IFS by $f_0$ and $f_1$. For any integer $m\ge 2$ let $g_m:\mathcal{I}^*\to\{0,1\}$ be defined by
$$g_m(\textbf{i})=
\begin{cases}
1 & \text{ if }\kappa(\textbf{i})\equiv 0\!\!\pmod m,\\
0 & \text{ otherwise.}
\end{cases}$$
Then define the subset $F_m\subset C$ by
$$F_m:=\Big\lbrace x=\lim_{n\to\infty}f_{g_m(\textbf{i}|_1)}\circ f_{g_m(\textbf{i}|_2)} \dots\circ f_{g_m(\textbf{i}|_n)}([0,1]):\textbf{i}\in\mathcal{I}\Big\rbrace.$$
Essentially, $F_m$ is generated by the binary tree for which every $m^{\text{th}}$ edge is labeled $1$ while all others are labeled $0$, see Figure \ref{C_3 example}. In the notation of the previous sections, $F_m=F(\omega)$ where $\omega=(0^{m-1}1)^\infty$.

Observe that $F_2=C$, since $g_2$ is precisely the usual labeling of the full binary tree. By contrast, for $m\geq 3$, $F_m$ has an interesting connection to golden mean-like shifts of finite type.

\begin{figure}
    \centering
    \begin{tikzpicture}
    \tikzstyle{level 1}=[level distance=10mm,sibling distance=3.5cm]
    \tikzstyle{level 2}=[level distance=15mm,sibling distance=1.5cm]
    \tikzstyle{level 3}=[level distance=15mm,sibling distance=1cm]
    \tikzstyle{solid node}=[circle,draw,inner sep=1.2,fill=black];
\node (0) [solid node]{ }
child{node[solid node]{ } 
    child{node[solid node]{ } 
        child{node[solid node]{ } edge from parent node[left]{0}} %Left Left Left
        child{node[solid node]{ } edge from parent node[right]{0}} %Left Left Right
        edge from parent node[left]{1}} %Left Left   
    child{node[solid node]{ } 
        child{node[solid node]{ } edge from parent node[left]{1}} %Left Right Left
        child{node[solid node]{ } edge from parent node[right]{0}} %Left Right Right
        edge from parent node[right]{0}} %Left Right
    edge from parent node[left]{0} }%Left
child{node[solid node]{ }
    child{node[solid node]{ } 
    child{node[solid node]{ } edge from parent node[left]{0}} %Right Left Left
    child{node[solid node]{ } edge from parent node[right]{1}}%Right Left Right
    edge from parent node[left]{0}} %Right Left    
    child{node[solid node]{ } 
    child{node[solid node]{ } edge from parent node[left]{0}}%Right Right Left
    child{node[solid node]{ } edge from parent node[right]{0}}%Right Right Right
    edge from parent node[right]{1}} %Right Right
    edge from parent node[right]{0}}; %Right
\end{tikzpicture}\qquad\quad
\begin{tikzpicture}[line width=1mm]
  \draw (0,0) -- (6,0);
  \draw (0,-1) -- (2,-1);
  \draw (0,-2) -- (0.667,-2);
  \draw (1.333,-2) -- (2,-2);
  \draw (0,-3) -- (0.222,-3);
  \draw (0.444,-3) -- (0.667,-3);
  \draw (1.333,-3) -- (1.556,-3);
\end{tikzpicture}
    \caption{An example tree with $m=3$ (left) and the third level approximation of $F_3$ (right).}
    \label{C_3 example}
\end{figure}
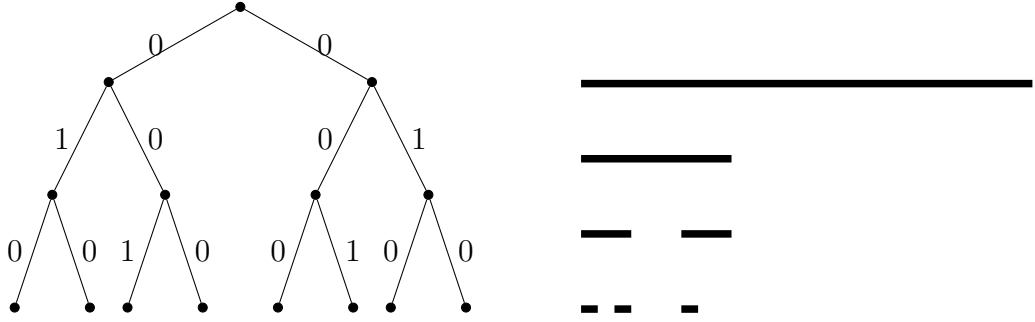

For $L\in\NN$, let $X_L\subseteq\{0,1\}^\NN$ be the one-sided subshift of finite type with forbidden words list
$$\mathcal{F}_L:=\big\lbrace10^k1:0\le k\le L-1\big\rbrace.$$
Thus, a sequence $\mathbf{i}\in\{0,1\}^\NN$ lies in $X_L$ if and only if any two consecutive $1$'s in $\mathbf{i}$ are separated by a string of at least $L$ consecutive $0$'s. It is well known that $X_L$ has entropy $\log\rho_L$, where $\rho_L\in(1,2)$ is the positive real root of $1+x^L=x^{L+1}$ (see \cite[Exercise~4.3.7]{Lind and Marcus}). In particular, $\rho_1=(1+\sqrt{5})/2$ is the golden ratio.

\begin{proposition}
For an integer $m\geq 3$, let $L$ be the integer such that $m\in\{2^{L+1}-1,2^{L+1},\dots,2^{L+2}-2\}$. Then
\begin{equation} \label{eq:C_m-projection}
\pi^{-1}[F_m]=\{\mathbf{j}=(j_1,j_2,\dots)\in X_L: j_1=j_2=\dots=j_L=0\},
\end{equation}
where $\pi$ is the natural projection of the symbolic space onto the Cantor set $C$, i.e.
\[
\pi(\mathbf{j}):=\lim_{n\to\infty} f_{j_1}\circ \dots \circ f_{j_n}(0) \qquad\mbox{for\ \ $\mathbf{j}=(j_1,j_2,\dots)\in\{0,1\}^\NN$}.
\]
Consequently,
\begin{equation} \label{eq:C_m-dimension}
\dim_H F_m=\frac{\log\rho_L}{-\log r}.
\end{equation}
(Recall that $r$ is the common contraction ratio of the maps $f_0$ and $f_1$.)
\end{proposition}

\begin{proof}
First, construct a directed graph $G$ whose vertices are labeled $0,1,\dots,m-1$ to represent the residual classes mod $m$.
Draw an edge from each vertex $j$ to $2j+1$ and $2j+2 \mod m$ (see Figure \ref{Directed graph example} for the case $m=6$).
Notice in particular that 
$$2(m-1)+1\equiv m-1 \!\!\pmod{m},\qquad \text{and} \qquad 2(m-1)+2\equiv 0 \!\!\pmod{m}.$$
Thus vertex $m-1$ admits a self loop, as well as an edge connected to vertex $0$. 
Furthermore, for any $k\le L+1$ and $2^{k}-1\le j\le \min\{2^{k+1}-2,m-1\}$, there is a walk of length $k$ from vertex $0$ to vertex $j$. This becomes clear when drawing $G$ as a tree rooted at $0$ and omitting incoming edges, as shown in Figure \ref{Directed graph as a tree}.

\begin{figure}
\centering
\begin{tikzpicture}[->,>=stealth',shorten >=1pt,auto,node distance=2.5cm, thick,scale=3]
  \tikzstyle{every state}=[minimum size=0pt,fill=none,draw=black,text=black]
\node[state] (0) {$0$};
\node[state] (1) [below right of=0] {$1$};
\node[state] (2) [below of=1] {$2$};
\node[state] (3) [below left of=2] {$3$};
\node[state] (4) [above left of=3] {$4$};
\node[state] (5) [above of=4] {$5$};
\path[->,every loop/.style={min distance=0mm, looseness=35}]
	(0) edge[->] (1)
	(0) edge[<->,semithick] (2)
	(1) edge[<->] (3)
	(1) edge[->] (4)
	(2) edge[->] (5)
	(2) edge[->] (0)
	(3) edge[->] (2)
	(4) edge[loop left,->] (4)
	(4) edge[->] (3)
	(5) edge[loop left,->] (5)
	(5) edge[->] (0);
\end{tikzpicture}
\caption{The directed graph $G$, shown for $m=6$. Each vertex $j$ has outgoing edges to $2j+1\!\! \mod m$ and $2j+2\!\! \mod m$. Vertex $m-1$ always has a self-loop and an outgoing edge to $0$.}
    \label{Directed graph example}
\end{figure}
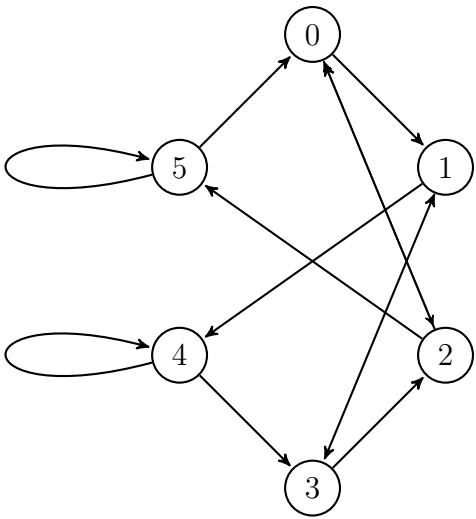

Now, by construction, every word in $\pi^{-1}[F_m]$ can be represented as an infinite walk in $G$ beginning at either vertex $1$ or $2$. That is, any step passing through vertices $1,\dots,m-1$ records a $0$, while a step through vertex $0$ records a $1$. Since $2^{L+1}-1\le m\le2^{L+2}-2$, we have a minimum of $L+1$ steps to walk from vertex $0$ back to itself in $G$.
Thus, no word in $\mathcal{F}_L$ can appear as a sub-word of any word in $\pi^{-1}[F_m]$.
Consequently, $\pi^{-1}[F_m]\subseteq X_L$.

To show the other inclusion, we need to show that for every integer $l\ge L+1$, there exists a walk in $G$ from vertex $0$ back to itself of length $l$ (without passing through $0$ along the way). 
This is clear for $l=L+1$, so assume $l>L+1$. If $m=2^{L+1}-1$, we take the $L$-step walk
$$0\to 2\to 6\to\dots\to 2^{L+1}-2=m-1,$$
then repeat the self loop at vertex $m-1$ for $l-L-1$ times, and finally take $m-1\to 0$ for the last step.
On the other hand, if $m\ge 2^{L+1}$, then as previously discussed we may reach vertex $m-1$ in $L+1$ steps, repeat the self loop for $l-L-2$ times, then take $m-1\to 0$ as the last step.
Therefore, every sequence in $X_L$ beginning with $0^L$ lies in $\pi^{-1}[F_m]$. 

This gives \eqref{eq:C_m-projection}. Finally, since the IFS $\{f_0,f_1\}$ satisfies the open set condition, it follows from standard arguments that the Hausdorff dimension of $F_m$ is given by the entropy of $X_L$ over $\log r^{-1}$ (see \cite[Lemma 16.4.2]{URM book} for example), yielding \eqref{eq:C_m-dimension}.
\end{proof}

\begin{figure}
\centering
\begin{tikzpicture}[thick] 
%every node/.style={draw,circle}]
\node[draw,circle] (0) at (4,4) {$0$};
\node[draw,circle] (1) at (2,3) {$1$};
\node[draw,circle] (2) at (6,3) {$2$};
\node[draw,circle] (3) at (1,2) {$3$};
\node[draw,circle] (4) at (3,2) {$4$};
\node[draw,circle] (5) at (5,2) {$5$};
\node[draw,circle] (6) at (7,2) {$6$};
\node (7) at (0,1) {};
\node (8) at (2,1) {};
\node (9) at (4,1) {};
\node (10) at (6,1) {};
\node (11) at (8,1) {};
\node (12) at (1,0.7) {\huge{$\vdots$}};
\node (13) at (4,0.7) {\huge{$\vdots$}};
\node (14) at (7,0.7) {\huge{$\vdots$}};
\draw[->] (0)--(1);
\draw[->] (0)--(2);
\draw[->] (1)--(3);
\draw[->] (1)--(4);
\draw[->] (2)--(5);
\draw[->] (2)--(6);
\draw[->] (3)--(7);
\draw[->] (3)--(8);
\draw[->] (4)--(8);
\draw[->] (4)--(9);
\draw[->] (5)--(9);
\draw[->] (5)--(10);
\draw[->] (6)--(10);
\draw[->] (6)--(11);
\end{tikzpicture}
\caption{Each vertex $j$ for $2^k-1\le j\le 2^{k+1}-2$ may be reached in exactly $k$ steps.}
\label{Directed graph as a tree}
\end{figure}
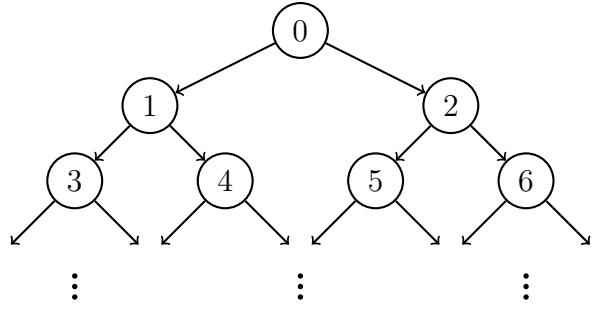

It may be of interest to compare the value $\dim_H F_m=-\log\rho_L/\log r$ with the bounds from Theorem \ref{thm:dimension-bounds} and Corollary \ref{cor:two-maps} for the random case with $N=M=2$ and $p=1/m$. We have made such a comparison in Table \ref{tab:comparisons}. Although there is no reason why the dimension of $F_m$ should lie between the bounds of Theorem \ref{thm:dimension-bounds} (since $F_m$ is but one of uncountably many realizations of the random subset $F(\omega)$), this does appear to be the case for all $m\geq 4$.

%\begin{figure}
%    \centering
%    \begin{tikzpicture}[xscale = 6, yscale = 6, domain = 0.0001:0.4999]
%    \draw[->] (-0.1,0) -- (0.6,0) node[right] {$p$};
%    \draw[->] (0,-0.1) -- (0,0.65) node[above] {$\text{dim}_H$};

%    \draw[color = red]
%    plot(\x,{(ln(2*\x)/(ln(1-\x)-ln(\x))) *%
%    (ln(-(ln(2*\x)/(ln(1-\x)-ln(\x)))) / ln(3))%
%    -%
%    (1+((ln(2*\x)/(ln(1-\x)-ln(\x))))) *%
%    (ln((1+((ln(2*\x)/(ln(1-\x)-ln(\x)))))) / ln(3))%
%    }) node[above] {Upper Bound};
%    \draw[color = blue] 
%    plot (\x,{-(ln(\x^2+(1-\x)^2))/ln(3)}) node[below] {Lower Bound};
%   \node [green] at (0.5,0.6309) {\textbullet};
%   \node [green] at (0.3333,0.438) {\textbullet};
%   \node [green] at (0.1429,0.2559) {\textbullet};
%   \node [green] at (0.0667,0.1815) {\textbullet};
%   \node [green] at (0.0313,0.1408) {\textbullet};
%   \node [green] at (0.0159,0.115) {\textbullet};
%    \end{tikzpicture}
%    \caption{The plot of $\dim_H C_m$ for $m=2,3,7,15,31$ and $63$ (green dots), compared with the upper and lower bounds from the previous section.}
%    \label{graph comparisons}
%\end{figure}

\begin{table}
\vspace{5mm}
\begin{tabular}{r|c|c|c|c}
$m$ & $p$ & lower bound & $\dim_H F_m$ & upper bound\\ \hline
2 & $1/2$ & .631 & .631 & .631\\
3 & $1/3$ & .535 & $.438^*$ & .618\\
4 & $1/4$ & .428 & .438 & .599\\
6 & $1/6$ & .296 & .438 & .569\\
7 & $1/7$ & .256 & .348 & .557\\
14 & $1/14$ & .130 & .348 & .503\\
15 & $1/15$ & .121 & .293 & .498\\
30 & $1/30$ & .061 & .256 & .450\\
\end{tabular}
\vspace{5mm}
\caption{Comparison of $\dim_H F_m$ with the upper and lower bounds from Theorem \ref{thm:dimension-bounds} and Corollary \ref{cor:two-maps} with $p=1/m$ and $r=1/3$, for select values of $m$. Note that $\dim_H F_m$ is constant on each interval $2^{L+1}-1\leq m\leq 2^{L+2}-2$, for $L\geq 1$. It appears that only the value for $m=3$ is outside the bounds.}
\label{tab:comparisons} 
\end{table}

\section*{Acknowledgments}

An earlier draft of this paper only considered the case of two maps with a binary tree. We are grateful to the referee for pushing us to investigate the more general case, and for several other insightful comments. We thank Sascha Troscheit for pointing out the connection with branching random walks, which helped visualize the problem. Allaart was partially supported by Simons Foundation grant \# 709869.


\begin{thebibliography}{13}

\bibitem{Barnsley}
{\sc M. F. Barnsley, J. E. Hutchinson} and {\sc \"O. Stenflo}, $V$-variable fractals: Fractals with partial self similarity {\em Adv. Math.} {\bf 218} (2008), no. 6, 2051--2088.

\bibitem{BGS}
{\sc I. Benjamini, O. Gurel-Gurevich} and {\sc B. Solomyak}, Branching random walk with exponentially decreasing steps, and stochastically self-similar measures. {\em Trans. Amer. Math. Soc.} {\bf 361} (2009), no. 3, 1625--1643.

\bibitem{Benjamini-Kesten}
{\sc I. Benjamini} and {\sc H. Kesten}, Percolation of arbitrary words in $\{0,1\}^\NN$, {\em Ann. Probab.} {\bf 23} 
(1995), 1024--1060.

\bibitem{Chen}
{\sc C. Chen}, A class of random Cantor sets, {\em Real Anal. Exchange}, {\bf 42} (2017), no. 1, 79--120.

\bibitem{Dekking and Karoly}
{\sc M. Dekking} and {\sc K. Simon} On the Size of the Algebraic Difference of Two Random Cantor Sets, {\em Random Structures  Algorithms} {\bf 32} (2007), no. 2, 205--222.

\bibitem{Falconer}
{\sc K. Falconer}, Random fractals, {\em Math. Proc. Cambridge Philos. Soc.} {\bf 100} (1986), no. 3, 559--582.

\bibitem{Falconer1}
{\sc K. Falconer}, Dimensions and measures of quasi self-similar sets, {\em Proc. Amer. Math. Soc.} {\bf 106} (1989), no.2, 543--554.

\bibitem{Falconer-techniques}
{\sc K. Falconer}, {\em Techniques in Fractal Geometry}, John Wiley and Sons, Ltd., Chichester, 1997.

\bibitem{Falconer2}
{\sc K. Falconer}, {\it Fractal Geometry: Mathematical Foundations and Applications}, 2nd Edition, Wiley and Sons, Ltd., Chichester, 2003.

\bibitem{Graf}
{\sc S. Graf}, Statistically self-similar fractals. {\em Probab. Theory Related Fields} {\bf 74} (1987), 357--392.

\bibitem{Hutchinson}
{\sc J.E. Hutchinson}, Fractals and self-similarity, {\em Indiana Univ. Math. J.} {\bf 30} (1981), 713--174.

\bibitem{Jarvenpaa}
{\sc E. J\"arvenp\"a\"a, M. J\"arvenp\"a\"a, A. K\"aenm\"aki, H. Koivusalo, \"O. Stenflo} and {\sc V. Suomala}, Dimensions of random affine code tree fractals. {\em Ergodic Theory Dynam. Systems} {\bf 34} (2014), 854--875.

\bibitem{Koivusalo}
{\sc H. Koivusalo} Dimension of uniformly random self-similar fractals, {\em Real Anal. Exchange} {\bf 39} (2014), no. 1, 73--90.

\bibitem{Larsson}
{\sc P. Larsson}, L’ensemble diff\'erence de deux ensembles de Cantor al\'eatoires, {\em C.R. Acad. Sci. Paris S\'er. I Math.} {\bf 310} (1990), no. 2, 735--738.

\bibitem{Lind and Marcus}
{\sc D. Lind} and {\sc B. Marcus}, {\it An Introduction to Symbolic Dynamics and Coding}, Cambridge University Press (1995).

\bibitem{Liu-Wu}
{\sc Y. Liu} and {\sc K. Wu}, A dimensional result for random self-similar sets, {\em Proc. Amer. Math. Soc.} {\bf 130} (2002), no. 7, 2125--2131.

%\bibitem{Lyons and Peres}
%{\sc R. Lyons and Y. Peres} {\it Probability on Trees and Networks}, Cambridge University Press (2016)

\bibitem{Mauldin and Williams}
{\sc R.D. Mauldin} and {\sc S.C. Williams}, Random recursive constructions: Asymptotic geometric and topological properties, {\em Trans. Amer. Math. Soc.} {\bf 295} (1986), 325--346

\bibitem{Shmerkin and Suomala}
{\sc P. Shmerkin} and {\sc V. Suomala}, Sets which are not tube null and intersection properties of random measures, {\em J. London Math. Soc.} {\bf 91} (2015), no. 2, 405--422.

\bibitem{URM book}
{\sc M. Urbański, M. Roy} and {\sc S. Munday}, {\em Non-invertible dynamical systems, Vol. 2}, De Gruyter, Berlin, 2022.

\end{thebibliography}
\end{document}